\documentclass{birkau}

\usepackage{amsmath,amssymb,url}
\usepackage{braket}

\usepackage{mathbbol}

\numberwithin{equation}{section}

\theoremstyle{plain}
\newtheorem{theorem}{Theorem}[section]
\newtheorem{lemma}[theorem]{Lemma}
\newtheorem{proposition}[theorem]{Proposition}
\newtheorem{corollary}[theorem]{Corollary}
\newtheorem{claim}[theorem]{Claim}

\theoremstyle{definition}
\newtheorem{definition}[theorem]{Definition}

\newtheorem{question}[theorem]{Question}
\newtheorem{example}[theorem]{Example}

\newenvironment{claimproof}[1][Proof of Claim]{\begin{proof}[#1]}{\end{proof}}

\newcommand\<{\langle}
\renewcommand\>{\rangle}
\newcommand\Wedge{\bigwedge}
\newcommand\suchthat{\mid}
\newcommand{\R}{\mathbb{R}}
\newcommand{\N}{\mathbb{N}}
\newcommand\includedin{\subseteq}
\newcommand\bbzero{\mathbb 0}
\newcommand\bbone{\mathbb 1}  
\DeclareMathOperator\EQ{\mathcal {EQ}}
\DeclareMathOperator\UM{\MM}
\newcommand\MM{{\mathcal M}}
\newcommand\fraisse{Fra\"\i ss\'e }
\newcommand\intersect{\cap}
\newcommand\includes{\supseteq}
\renewcommand\AA{{\mathcal A}}

\begin{document}


\title[$\Lambda$-ultrametric spaces and
lattices of equivalence relations]{$\Lambda$-ultrametric spaces and
lattices of equivalence relations}

\author[Samuel Braunfeld]{Samuel Braunfeld}
\address{Department of Mathematics\\
University of Maryland, College Park\\Maryland 20742\\USA}
\urladdr{math.umd.edu/~sbraunf}
\email{sbraunf@math.umd.edu}



\begin{abstract}
For a finite lattice $\Lambda$, $\Lambda$-ultrametric spaces have, among other reasons, appeared as a means of constructing structures with lattices of equivalence relations embedding $\Lambda$. This makes use of an isomorphism of categories between $\Lambda$-ultrametric spaces and structures equipped with certain families of equivalence relations. We extend this isomorphism to the case of infinite lattices. We also pose questions about representing a given finite lattice as the lattice of $\emptyset$-definable equivalence relations of structures with model-theoretic symmetry properties.
\end{abstract}

\maketitle

\section{Introduction} 

Our goal is to present an isomorphism between two categories, which arose in simplified form in the course of constructing certain homogeneous (Definition \ref{def:homog}) structures in \cite{Lat}.

The first category has objects consisting of a set equipped with a family of equivalence relations forming a lattice $\Lambda$, or substructures thereof. However, we initially define its objects as structures consisting of 
of a set equipped with a family of equivalence relations,
closed under taking meets in the lattice of all equivalence
relations on the set, and labeled by the elements
of a fixed lattice $\Lambda$ in such a way that
the map from $\Lambda$ to the lattice of equivalence relations is a homomorphism.

The objects of the second category are lattice-valued
metric spaces (also known as generalized ultrametric spaces) with metric valued in a certain
complete lattice $\Phi(\Lambda)$ containing $\Lambda$.

 In \cite{Lat}, the structures of interest were the families of equivalence relations, but the metric spaces were technically easier to work with. As $\Lambda$ was assumed to be finite, almost everything trivializes; in particular $\Phi(\Lambda) = \Lambda$. The correspondence is then roughly as follows. Starting with a structure equipped with a suitable family of equivalence relations, one defines $d(x, y)$ to be the finest equivalence relation that holds between $x$ and $y$. From this distance, one can easily recover all the equivalence relation information.

The general correspondence is largely a matter of taking care in
setting up the correct categories and in the choice of the map $\Phi$---we want the lattice of filters on $\Lambda$, rather than the usual Dedekind-MacNeille completion.

Generalized ultrametric spaces were first investigated by Priesse-Crampe and Ribenboim \cite{GUM}, who seem primarily interested in applications to valued fields. In \cite{PCR} and \cite{GUM}, the correspondence with certain structures of equivalence relations is developed to some extent, but not to the point where passing to $\Phi(\Lambda)$ becomes necessary. The correspondence also appears in restricted contexts in \cite[\S 6.2]{Berg} and \cite{PT}, in order to produce embeddings of a given lattice into the lattice of all equivalence relations on a given set. 

The paper is divided into three further sections.  
First, we present some preliminary material on the filter lattice $\Phi(\Lambda)$, along with an example motivating the passage to it. As this lattice has appeared several times before, these results are likely standard. The next section presents the isomorphism of categories that is our main result. Finally, we close by mentioning the motivating model-theoretic context and some related open questions.

\section{The filter lattice}

\begin{definition}
For $\Lambda$ a lattice, let $\Phi(\Lambda)$
denote the \emph{filter lattice}, defined as follows.

\begin{enumerate}
\item Elements are filters in $\Lambda$
(i.e., nonempty subsets, closed under meet and closed upwards).
\item The partial order on $\Phi(\Lambda)$
is defined as reverse inclusion.
\end{enumerate}

For any $L \subset \Lambda$, we write
$\<L\>$ for the filter generated by $L$,
which is the upward closure of the set of elements
of the form
$\Wedge L_0$
for $L_0 \subset L$ finite.

For $\lambda\in \Lambda$, let $\widehat \lambda$
denote the principal filter $\<\lambda\>=\set{x\in \Lambda\suchthat x\ge \lambda}$.
\end{definition}

The following example shows how the correspondence may fail without the passage to $\Phi(\Lambda)$.

\begin{example}
Let $\Lambda$ be the interval $[0,1]$ in $\R$ with
the usual ordering. Suppose we have a structure equipped with a suitable family of equivalence relations labeled by elements of $\Lambda$. If  $xE_\lambda y$ for $\lambda \in (0,1]$, then there is no finest equivalence relation that holds between $x$ and $y$, so we cannot assign a distance as in the case where $\Lambda$ is finite.

For the given $\Lambda$, $\Phi(\Lambda)$ is isomorphic
to 
\[
([0,1]\times \{0,1\})\setminus \{(1,1)\}
\]
with the
lexicographic ordering, with $\widehat\lambda$ corresponding to $(\lambda,0)$
and with $\ \widehat\lambda\setminus \{\lambda\}$ corresponding to $(\lambda,1)$.
As we do not allow empty filters, we leave off $(1,1)$.

For the $x, y$ given in the first paragraph, we can now assign $d(x, y) = (0,1) \in \Phi(\Lambda)$.
\end{example}

\begin{lemma}
Let $\Lambda$ be a bounded lattice. Then the following hold.

\begin{enumerate}
\item $\Phi(\Lambda)$ is a complete lattice with
join given by intersection, and meet by taking the
filter generated by the union of a given set of filters.

\item The map
$\phi \colon \Lambda\to \Phi(\Lambda)$ defined by
\[
\lambda\mapsto \widehat \lambda
\]
is an isomorphic embedding with respect
to the lattice operations.

\item If $\lambda\in\Lambda$, $L\includedin \Lambda$,
and $\widehat \lambda=\Wedge \set{\widehat \ell | \ell \in L}$ in $\Phi(\Lambda)$,
then $\lambda$ is the meet of some finite subset of $L$, in $\Lambda$.
\end{enumerate}
\end{lemma}

\begin{proof}
(1) The order is reverse inclusion, so the rules for
join and meet are clear. For completeness, the only point to check is that the intersection of arbitrarily many filters is always non-empty. As $\Lambda$ has a maximum element, this holds.

(2) It suffices to check that the map $\phi$ 
preserves
meet and join, or in other words
\begin{align*}
\widehat \lambda_1\vee \widehat \lambda_2&=\widehat{\lambda_1\vee \lambda_2}\\
\widehat \lambda_1\wedge \widehat \lambda_2&=\widehat{\lambda_1\wedge \lambda_2}
\end{align*}

More explicitly, this reads as follows.
\begin{align*}
x\ge \lambda_1,\lambda_2&\iff x\ge \lambda_1\vee \lambda_2\\
x\in \< \widehat \lambda_1,\widehat \lambda_2\>&
\iff x\ge \lambda_1\wedge \lambda_2
\end{align*}

Both points are clear.

(3) If $\widehat \lambda=\Wedge \set{\widehat \ell | \ell \in L}$ then $\lambda\in \< L\>$ and therefore $\lambda\in \< L_0\>$ for some finite set $L_0$. The claim follows.
\end{proof}

The third point in the lemma above seems to capture the key abstract property of the construction. However, it is not directly used in what follows, as we prefer to work with the explicit definition of $\Phi(\Lambda)$.

\section{The correspondence}

\begin{definition}
Suppose $\Lambda$ is a bounded lattice with minimal element $\bbzero$ and maximal element $\bbone$.
\begin{enumerate}
\item
Let $\EQ_\Lambda$ denote the 
category whose objects are the models in the language
\[
\set{E_\lambda\suchthat \lambda \in \Lambda}
\]
for which the following hold.
\begin{itemize}
\item All $E_\lambda$ are equivalence relations.
\item The map from $\lambda\to E_\lambda$
preserves meets; in particular, the set of relations
$\set{E_\lambda\suchthat \lambda\in \Lambda}$
is closed under intersection.
\item $E_{\bbzero}$ is equality, and $E_{\bbone}$ is the trivial relation.
\end{itemize}
We take embeddings as morphisms.

\item
A {\it $\Lambda$-ultrametric space} 
is a structure of the form
$(X,d)$ for which 
\begin{itemize}
\item $d \colon X^2\to \Lambda$ is a symmetric function.
\item $d(x,y)=0$ iff $x=y$.
\item 
$d(x,y)\le d(x,z)\vee d(y,z)$ for all $x,y,z$
(Triangle Inequality).
\end{itemize}

\item Let $\UM_{\Lambda}$ be the category of
$\Lambda$-ultrametric spaces,
with isometric embeddings as morphisms.
\end{enumerate}
\end{definition}

In the model-theoretic context, we are primarily concerned with $\emptyset$-definable equivalence relations. While these form a lattice in an $\omega$-categorical structure (Definition \ref{def:omcateg}), in general they only need form a lower semi-lattice. It would thus seem natural to define the category $\EQ_\Lambda$ 
for any lower semi-lattice $\Lambda$ 
and to define the category $\UM_\Lambda$
for any upper semi-lattice $\Lambda$. 
But it is not clear what could replace the correspondence that we aim at below.

\begin{example}
If $\Lambda$ is a chain, then $\Lambda$-ultrametric spaces are ultrametric spaces in the usual sense.
\end{example}

\begin{example}
Let $\Lambda$ be a lattice and
let $A$ be a set with at least two elements.
Define $E_\lambda$ to be equality for $\lambda\in \Lambda\setminus \{\bbone\}$, and let
$E_\bbone$ be trivial. Then $(A,\set{E_\lambda\suchthat \lambda\in \Lambda})$ belongs
to $\EQ_\Lambda$. The map sending
$\lambda$ to $E_\lambda$ is a semi-lattice
homomorphism, but will not be a lattice
homomorphism unless $\bbone$ is 
join-irreducible. 

This example arises naturally as a substructure
of any structure in $\EQ_\lambda$ having
a pair of elements that are not related by
any of the relations $E_\lambda$ for $\lambda<\bbone$. For example, there is a natural representation of the Boolean algebra with $n$
atoms by equivalence relations on the set $\{0,1\}^n$, in which the co-atoms correspond to the relations 
$E_i(x,y)\iff x_i=y_i$. Intersections of co-atoms in the Boolean algebra correspond to intersections of these relations, and so specify agreement in multiple coordinates. We may then consider
the substructure induced on the constant sequences, in which two elements will be related by a relation other than $E_\bbone$ only if they agree in some coordinate and thus must be equal, i.e. $E_\bbzero$-related.
\end{example}

Now the main point is to consider the
relationship between the categories
$\EQ_\Lambda$ and $\UM_{\Phi(\Lambda)}$.
The former is the category that interests us, while
the latter is easier to work with in the context
of the \fraisse theory of amalgamation classes.

\begin{theorem}\label{Prop:Equivalence}
Let $\Lambda$ be a bounded lattice.
Then the categories 
$\EQ_\Lambda$ 
and
$\UM_{\Phi(\Lambda)}$ 
are canonically isomorphic.
\end{theorem}
\begin{proof}
Given 
$\AA
=(A,\set{E_\lambda\suchthat \lambda\in \Lambda})$ in $\EQ_\Lambda$,
let  
$m(\AA)$ be
$(A,d)$, where 
\[
d \colon A\times A\to \Phi(\Lambda)
\]
is defined by 
\[
d(x,y)=\bigwedge \set{\widehat \lambda
\suchthat \mbox{$\lambda\in \Lambda$ and $E_\lambda(x,y)$ in $\AA$}}
\]

In the reverse direction, 
for $\MM=(M,d)$ a $\Phi(\Lambda)$-ultrametric space,
let $e(\MM)=(M,\set{E_\lambda\suchthat \lambda\in \Lambda})$ where $E_\lambda$ is defined by
\[
E_{\lambda}(x,y)\iff \lambda \in d(x,y)
\]

We will show that $m$ and $e$ 
give a pair of mutually
inverse bijections between the objects, and
that these bijections
give rise to an isomorphism between
the categories.

There are a number of points to be checked. Let
us begin by listing them all.

\begin{enumerate}
\item $m \colon \EQ_\Lambda\to \UM_{\Phi(\Lambda)}$
\item $e\colon \UM_{\Phi(\Lambda)}\to \EQ_\Lambda$
\item At the level of objects, the maps $m,e$ 
are mutually inverse.
\item If $f\colon A\to B$ is a morphism in one category,
then $f\colon A'\to B'$ is also a morphism between
the corresponding objects in the other category.
\end{enumerate}

\begin{claim}
If $\AA\in \EQ_\Lambda$, then
$m(\AA)\in \UM_{\Phi(\Lambda)}$.
\end{claim}

\begin{claimproof}
As $\Phi(\Lambda)$ is complete, the definition
of $d(x,y)$ makes sense, and
\[
d\colon A\times A\to \Phi(\Lambda)
\]

We first make the definition of $d$ more explicit.
Note that for $x,y\in A$ the set $\set{\lambda\in \Lambda\suchthat E_\lambda(x,y)}$ is a filter,
since the map from $\lambda$ to $E_\lambda$
preserves meets.
We have
\[
d(x,y)=
\Wedge \set{\widehat \lambda\suchthat E_\lambda(x,y)}
=\set{\lambda\suchthat E_\lambda(x,y)}
\]

Symmetry of $d$ is clear.

For the triangle inequality, take $x,y,z\in A$ 
and note that
\begin{align*}
d(x,z)\vee d(y,z)&=
d(x,z)\intersect d(y,z)
=\set{\lambda\suchthat E_\lambda(x,z), 
E_\lambda(y,z)} 
\\
&\includedin \set{\lambda\suchthat E_\lambda(x,y)}
=d(x,y)
\end{align*}

Our last point is that $d(x,y)=\bbzero$ iff $x=y$.

For $x\in A$ we have $E_\lambda(x,x)=\Lambda=\widehat \bbzero=\bbzero$.

For the converse, 
if $d(x,y)=\bbzero$ in $\Phi(\Lambda)$ this
means
\begin{align*}
\set{\lambda\suchthat E_\lambda(x,y)}&=\Lambda
\end{align*}
and thus $E_{\bbzero}(x,y)$ holds. As we assume $E_{\bbzero}$ is equality, we find $x=y$.
\end{claimproof}

\begin{claim}
If $\MM\in \UM_{\Phi(\Lambda)}$ then
$e(\MM)\in \EQ_\Lambda$.
\end{claim}

\begin{claimproof}
We check first that the 
relations $E_\lambda$ are equivalence relations.

\begin{itemize}
\item Reflexivity: In $\Phi(\Lambda)$, 
$\bbzero=\Lambda$.
\item Symmetry: $d$ is symmetric.
\item Transitivity:  The triangle inequality 
may be written more explicitly in the following form.
\begin{align*}
d(x,y)&\includes d(x,z)\intersect d(y,z)
\end{align*}
It is then clear that each relation $E_\lambda$
is transitive.
\end{itemize}

The second point to be checked is preservation of meets: 
\begin{align*}
E_{\lambda\wedge \lambda'}&=E_\lambda \intersect E_{\lambda'}
\end{align*}

Now
\begin{align*}
E_{\lambda\wedge \lambda'}(x,y)&\iff \lambda\wedge \lambda'\in d(x,y)\\
&\iff \lambda,\lambda'\in d(x,y)
\end{align*}
since $d(x,y)$ is a filter in $\Lambda$.

The last point is that $E_\bbzero$ and $E_\bbone$ 
are as intended.

It is easy to see that $E_\bbzero(x,y)$ holds
iff $d(x,y)=\bbzero$ (in $\Phi(\Lambda)$), and
this is equivalent to $x=y$.

And $E_\bbone(x,y)$ holds iff $\bbone \in d(x,y)$
which is always the case.
\end{claimproof}

\begin{claim}
For $\AA\in \EQ_\Lambda$ and $\MM\in \UM_{\Phi(\Lambda)}$, we have
\begin{align*}
e(m(\AA))&=\AA\\
m(e(\MM))&=\MM
\end{align*}
\end{claim}

\begin{claimproof}
We know
\begin{align*}
\lambda \in d(x,y)&\iff E_\lambda(x,y) \mbox{ when $\MM=m(\AA)$ and $x,y\in A$}\\
E_\lambda(x,y)&\iff \lambda\in d(x,y) \mbox{ when $\AA=e(\MM)$ and $x,y,\in M$}
\end{align*}

Two applications of these rules will clearly bring us back where we started.

\begin{claim}
If $\MM_i=e(\AA_i)$ for $i=1,2$, then
a map $f\colon A_1\to A_2$ will be an embedding
if and only if it is a $\Phi(\Lambda)$-isometry.
\end{claim}

Let us compare the two properties.

\begin{align*}
E_\lambda(x,y)&\iff E_\lambda(f(x),f(y)) 
\mbox{ for $x,y\in A_1$ and $\lambda\in \Lambda$}\\
d(x,y)&=d(f(x),f(y))
\end{align*}

Recalling that $\lambda\in d(x,y)$ iff 
$E_\lambda(x,y)$ holds, and the same for
$f(x),f(y)$, the claim follows.
\end{claimproof}
\end{proof}

\section{Model-theoretic context}
In this section, we outline the model-theoretic context, namely homogeneous structures in the sense of \fraisse theory, in which this correspondence arose, and mention some related open problems.

\begin{definition} \label{def:homog}
A countable structure $M$ is \emph{homogeneous} if every partial isomorphism between finitely-generated substructures extends to an automorphism of $M$.
\end{definition}

\begin{definition} \label{def:omcateg}
A countable structure $M$ is $\omega$-\emph{categorical} if it is the unique countable model of its first-order theory, or equivalently if Aut$(M)$ has only finitely many orbits in its diagonal action on $M^n$ for each $n \in \N$.
\end{definition}

Under the assumption of a finite relational language, homogeneity implies $\omega$-categoricity. Lying at the intersection of fields such as model theory, permutation groups, and combinatorics, these properties interact with a wide range of subjects. For example, in constraint satisfaction problems, they are crucial in allowing the universal algebraic approach from finite domains to be generalized to infinite domains \cite{Bod}.

The correspondence was used in the course of constructing homogeneous structures from $\EQ_\Lambda$, which would serve as the lattice of $\emptyset$-definable equivalence relations for the homogeneous structures of interest. This proceeds via combinatorial analysis of the finite substructures, and as the passage to substructures is simpler in $\UM_{\Phi(\Lambda)}$, it was preferable to work in that category.

The next corollaries could be proven directly, but are immediate with the correspondence in hand.

\begin{corollary}
The categories $\EQ_\Lambda$ and
$\UM_{\Lambda}$ are closed under passage
to substructures.
\end{corollary}
\begin{proof}
This is clear in the case of $\UM_\Lambda$,
and if we apply it to $\UM_{\Phi(\Lambda)}$
and use Claim 4 of Theorem \ref{Prop:Equivalence},
in the case of inclusion maps, the claim follows.
\end{proof}

\begin{corollary}
Let $\AA\in \EQ_\Lambda$ and $\MM=m(\AA)$.
Then $\AA$ is a homogeneous structure iff
$\MM$ is a homogeneous $\Phi(\Lambda)$-ultrametric
space.
\end{corollary}
\begin{proof}
Substructure and isomorphism 
correspond between these classes by Theorem \ref{Prop:Equivalence}.

Suppose $\AA$ is homogeneous. Let $X, Y \subset \MM$ be finite such that $X \cong Y$. Then $e(X), e(Y) \subset \AA$ are finite and isomorphic, and as $\AA$ is homogeneous, there is an automorphism $\sigma$ of $\AA$ sending $e(X)$ to $e(Y)$. Then $m \circ \sigma \circ e$ is an automorphism of $\MM$ sending $X$ to $Y$.

The proof of the other direction is essentially the same, swapping the roles of $e$ and $m$.
\end{proof}

We now turn to open questions, prompted by the following result of \cite{Lat}.

\begin{proposition}
Let $\Lambda$ be a finite distributive lattice. Then there is a homogeneous $\Lambda$-ultrametric space with lattice of $\emptyset$-definable equivalence relations isomorphic to $\Lambda$.
\end{proposition}

\begin{question}
What finite lattices are realizable as the lattice of $\emptyset$-definable equivalence relations of some homogeneous (resp. $\omega$-categorical) structure?
\end{question}

\begin{question}
Classify the homogeneous $\Lambda$-ultrametric spaces, for finite $\Lambda$.
\end{question}

In \cite{Lat}, it was also shown that distributivity is necessary under additional hypotheses.

\begin{definition}
Let $\Lambda$ be the lattice of $\emptyset$-definable equivalence relations on $M$. Then $\Lambda$ has the \emph{infinite index property} if whenever $E < F$, each $F$-class splits into infinitely many $E$-classes.
\end{definition}

\begin{proposition}
Let $M$ be a homogeneous structure with finite lattice of $\emptyset$-definable equivalence relations $\Lambda$. Suppose $\Lambda$ has the infinite index property and the reduct of $M$ to the language of $\emptyset$-definable equivalence relations is homogeneous. Then $\Lambda$ is distributive.
\end{proposition}

However, \cite[\S 5.2]{thesis} contains an example without the infinite index property whose lattice of $\emptyset$-definable equivalence relations is the non-distributive lattice $M_3$. This example may be viewed as taking the equivalence relations to be parallel classes of lines in the affine plane over $\mathbb{F}_2$.


\subsection*{Acknowledgment}

I thank Gregory Cherlin for helpful discussions. The final publication is available at springerlink.com. \\ \url{link.springer.com/article/10.1007/s00012-019-0606-4}


\end{document}